\documentclass[reqno,a4paper]{amsart}
\usepackage{amssymb}
\usepackage{amsmath}
\usepackage{amsfonts}

\newtheorem{theorem}{Theorem}[section]
\theoremstyle{plain}

\newtheorem{corollary}{Corollary}[section]

\newtheorem{lemma}{Lemma}[section]

\numberwithin{equation}{section}

\input{tcilatex}

\begin{document}
\title[]{On Better Approximation Order for the Nonlinear $q$-Bernstein
Operator of Maximum Product Kind}
\author{SEZ\.{I}N \c{C}\.{I}T}
\address{Department of Mathematics, Faculty of Science, Gazi University,
Ankara, Turkey}
\email{sezincit@gazi.edu.tr}
\author{OG\"{U}N DO\u{G}RU}
\address{Department of Mathematics, Faculty of Science, Gazi University,
Ankara, Turkey}
\email{ogun.dogru@gazi.edu.tr}
\thanks{}
\urladdr{}
\date{}
\thanks{}

\begin{abstract}
Nonlinear $q$-Bernstein operator of max-product kind was introduced and its
approximation order was examined, and the order of approximation was found
to be $1/\sqrt{\left[ n\right] _{q}}$ by Duman in \cite{q bernstein}. In
this paper, we found a better order of approximation for this operator.
\end{abstract}

\maketitle

\section{Introduction}

In recent years, the nonlinear Bernstein operator of max-product kind has
been introduced, and some approximation properties have been examined by
Bede et al. \cite{max-prod-berns}, \cite{max-prod-berns-szazs}. About
nonlinear max-product type operators details can be found in \cite%
{max-prod-shepard}, \cite{approx by pseudo}, \cite{max-prod kitap} and \cite%
{max-prod-Szasz}. These type of nonlinear max-product operators are defined
in the semi-ring structure called as the maximum product algebra by
replacing the maximum with the sum as%
\begin{equation*}
B_{n}^{(M)}\left( f\right) \left( x\right) =\frac{\bigvee%
\limits_{k=0}^{n}p_{n,k}\left( x\right) f\left( \frac{k}{n}\right) }{%
\bigvee\limits_{k=0}^{n}p_{n,k}\left( x\right) },
\end{equation*}%
where $n\in 
\mathbb{N}
,$ $f\in C\left[ 0,1\right] ,$ $x\in \left[ 0,1\right] ,$ and $p_{n,k}\left(
x\right) =\binom{n}{k}x^{k}\left( 1-x\right) ^{n-k}.$ The order of
approximation for $B_{n}^{(M)}(f)\left( x\right) $ can be found in \cite%
{max-prod-berns} by means of the modulus of continuity as $\omega \left( f;1/%
\sqrt{n}\right) $. In \cite{better bernstein} we find that the order of
approximation of this operator can be obtained as $1/n^{1-\frac{1}{\alpha }%
}, $ where $\alpha =2,3,...$ . So, we show that the order of approximation
for this operator given in \cite{max-prod-berns} can be improved for big
enough values of $\alpha $.

The classical Bernstein polynomials based on $q$-integers was introduced in 
\cite{Philips q}, \cite{Lupas} as 
\begin{equation*}
B_{n,q}(f)\left( x\right) =\sum_{k=0}^{n}\left[ 
\begin{array}{c}
n \\ 
k%
\end{array}%
\right] _{q}f\left( \frac{\left[ k\right] _{q}}{\left[ n\right] _{q}}\right)
x^{k}\prod\limits_{s=0}^{n-k-1}\left( 1-q^{s}x\right)
\end{equation*}%
where $n\in 
\mathbb{N}
,$ $f\in C\left[ 0,1\right] ,$ $x\in \left[ 0,1\right] $ and $q\in \left( 0,1%
\right] $.

First of all, we recall some definations for the concept of $q$-integer. The
definition of the $q$-integer for any nonnegative integer $n$ is 
\begin{equation*}
\left[ n\right] _{q}=\left\{ 
\begin{array}{ccc}
\frac{1-q^{n}}{1-q} & \text{if} & q\in \left( 0,1\right) \\ 
n & \text{if} & q=1%
\end{array}%
\right. .
\end{equation*}%
The $q$-factorial is defind as 
\begin{equation*}
\left[ n\right] _{q}!=\left\{ 
\begin{array}{ccc}
\left[ n\right] _{q}...\left[ 2\right] _{q}\left[ 1\right] _{q} & \text{if}
& n=1,2,... \\ 
1 & \text{if} & n=0%
\end{array}%
\right.
\end{equation*}%
and the $q$-binomial coefficient is defind as%
\begin{equation*}
\left[ 
\begin{array}{c}
n \\ 
k%
\end{array}%
\right] _{q}=\frac{\left[ n\right] _{q}!}{\left[ k\right] _{q}!\left[ n-k%
\right] _{q}!}.
\end{equation*}%
In \cite{q bernstein}, Duman introduced nonlinear $q$-Bernstein operator as%
\begin{equation}
B_{n,q}^{(M)}(f)\left( x\right) =\frac{\bigvee\limits_{k=0}^{n}p_{n,k}\left(
x;q\right) f\left( \frac{\left[ k\right] _{q}}{\left[ n\right] _{q}}\right) 
}{\bigvee\limits_{k=0}^{n}p_{n,k}\left( x;q\right) },  \label{q-Berns op}
\end{equation}%
where $n\in 
\mathbb{N}
,$ $f\in C\left[ 0,1\right] ,$ $x\in \left[ 0,1\right] $, $q\in \left(
0,1\right) $ and $p_{n,k}\left( x;q\right) =\left[ 
\begin{array}{c}
n \\ 
k%
\end{array}%
\right] _{q}x^{k}\prod\limits_{s=0}^{n-k-1}\left( 1-q^{s}x\right) .$ And
then he obtained the order of approximation for this operator by means of $1/%
\sqrt{\left[ n\right] _{q}}.$

In this paper, our aim is to improve this order of approximation for the
nonlinear $q$-Bernstein operator similarly as \cite{better bernstein}.

\section{Construction of the Operators}

In this section, let's recall the concepts in \cite{q bernstein}. Over the
set of $%
\mathbb{R}
_{+}$, we consider the operations $\bigvee $\ (maximum) and $\cdot $
(product). Then $\left( 
\mathbb{R}
_{+},\vee ,\cdot \right) $ has a semirings structure and it is called as
maximum product algebra (see, for instance, \cite{max-prod-berns-szazs}, 
\cite{approx by pseudo}). Let us consider%
\begin{equation*}
C_{+}\left[ 0,1\right] =\left\{ f:\left[ 0,1\right] \rightarrow \mathbb{R}%
_{+}:f\text{ continuous on }\left[ 0,1\right] \right\} .
\end{equation*}%
Since $B_{n,q}^{(M)}(f)\left( 0\right) =f\left( 0\right) $ for all $n,$ in
this part, we will consider $x>0$ in the notations, proofs and statements of
the all approximation result. $B_{n,q}^{(M)}(f)\left( x\right) $ is a
positive operator because of $f\in C_{+}\left[ 0,1\right] ,$ and $%
p_{n,k}\left( x;q\right) >0$ for all $x\in \left[ 0,1\right] .$ But, it is
not linear over $C_{+}\left[ 0,1\right] .$ Now, let us consider $f,g\in C_{+}%
\left[ 0,1\right] ,$ by the definition we see that, 
\begin{equation}
f\leq g\Longrightarrow B_{n,q}^{(M)}(f)\left( x\right) \leq
B_{n,q}^{(M)}(g)\left( x\right) .  \label{intro 1}
\end{equation}%
So, $B_{n,q}^{(M)}(f)\left( x\right) $ is increasing with respect to $f.$ In
adition, for any $f,g\in C_{+}\left[ 0,1\right] $ we get 
\begin{equation}
B_{n,q}^{(M)}(f+g)\left( x\right) \leq B_{n,q}^{(M)}(f)\left( x\right)
+B_{n,q}^{(M)}(g)\left( x\right) .  \label{intro 2}
\end{equation}%
Let $\omega \left( f,\delta \right) $, $\delta >0,$ denote the classical
modulus of continuity of $f\in C_{+}\left[ 0,1\right] $ defined by 
\begin{equation*}
\omega \left( f,\delta \right) =\max\limits_{x,y\in \left[ 0,1\right]
}\left\vert f\left( x\right) -f\left( y\right) \right\vert .
\end{equation*}%
Using (\ref{intro 1}), (\ref{intro 2}), and also applying Corollary $3$ in 
\cite{max-prod-berns-szazs} or Corollary $2.3$ in \cite{max-prod-berns}, we
have the following:

\begin{corollary}
\label{Corollary 1} \cite{q bernstein} For all $f\in C_{+}\left[ 0,1\right]
, $ $n\in \mathbb{N},$ $x\in \left[ 0,1\right] $ and $q\in \left( 0,1\right) 
$ we have 
\begin{equation*}
\left\vert B_{n,q}^{(M)}(f)\left( x\right) -f(x)\right\vert \leq 2\omega
\left( f,\delta _{n}\left( x;q\right) \right) ,
\end{equation*}%
where 
\begin{equation*}
\delta _{n}\left( x;q\right) :=B_{n,q}^{(M)}(\varphi _{x})\left( x\right) 
\text{ with }\varphi _{x}\left( t\right) =\left\vert t-x\right\vert .
\end{equation*}
\end{corollary}

\section{Auxiliary Results}

Let us define the following expression similar as in \cite{q bernstein}.

For each $k,j\in \left\{ 0,1,...,n\right\} $ and $x\in \left[ \frac{\left[ j%
\right] _{q}}{\left[ n+1\right] _{q}},\frac{\left[ j+1\right] _{q}}{\left[
n+1\right] _{q}}\right] $,%
\begin{equation*}
M_{k,n,j}\left( x;q\right) =\frac{p_{n,k}\left( x;q\right) \left\vert \frac{%
\left[ k\right] _{q}}{\left[ n\right] _{q}}-x\right\vert }{p_{n,j}\left(
x;q\right) }\text{ and }m_{k,n,j}\left( x;q\right) =\frac{p_{n,k}\left(
x;q\right) }{p_{n,j}\left( x;q\right) }.
\end{equation*}%
If $k\geq j+1$ then we have%
\begin{equation*}
M_{k,n,j}\left( x;q\right) =\frac{p_{n,k}\left( x;q\right) \left( \frac{%
\left[ k\right] _{q}}{\left[ n\right] _{q}}-x\right) }{p_{n,j}\left(
x;q\right) },
\end{equation*}%
and if $k\leq j-1$ then we have%
\begin{equation*}
M_{k,n,j}\left( x;q\right) =\frac{p_{n,k}\left( x;q\right) \left( x-\frac{%
\left[ k\right] _{q}}{\left[ n\right] _{q}}\right) }{p_{n,j}\left(
x;q\right) }.
\end{equation*}%
Also for $k,j\in \left\{ 0,1,...,n\right\} $ and $x\in \left[ \frac{\left[ j%
\right] _{q}}{\left[ n+1\right] _{q}},\frac{\left[ j+1\right] _{q}}{\left[
n+1\right] _{q}}\right] $

if $k\geq j+2$ then we have%
\begin{equation*}
\overline{M}_{k,n,j}(x;q)=\frac{p_{n,k}\left( x;q\right) \left( \frac{\left[
k\right] _{q}}{\left[ n+1\right] _{q}}-x\right) }{p_{n,j}\left( x;q\right) },
\end{equation*}%
if $k\leq j-2$ then we have%
\begin{equation*}
\underline{M}_{k,n,j}(x;q)=\frac{p_{n,k}\left( x;q\right) \left( x-\frac{%
\left[ k\right] _{q}}{\left[ n+1\right] _{q}}\right) }{p_{n,j}\left(
x;q\right) }.
\end{equation*}%
At this point, let us recall the following two lemmas:

\begin{lemma}
\cite{q bernstein} \label{Lemma 2} Let $q\in \left( 0,1\right) $ and $x\in %
\left[ \frac{\left[ j\right] _{q}}{\left[ n+1\right] _{q}},\frac{\left[ j+1%
\right] _{q}}{\left[ n+1\right] _{q}}\right] .$ Then we have

$(i)$ for all $k,j\in \left\{ 0,1,...,n\right\} $ with $k\geq j+2$, 
\begin{equation*}
\overline{M}_{k,n,j}\left( x;q\right) \leq M_{k,n,j}\left( x;q\right) \leq
\left( 1+\frac{2}{q^{n+1}}\right) \overline{M}_{k,n,j}\left( x;q\right) ,
\end{equation*}

$(ii)$ for all $k,j\in \left\{ 0,1,...,n\right\} $ with $k\leq j-2$, 
\begin{equation*}
M_{k,n,j}\left( x,q\right) \leq \text{$\underline{M}$}_{k,n,j}\left(
x;q\right) \leq \left( 1+\frac{2}{q^{n}}\right) M_{k,n,j}\left( x;q\right) .
\end{equation*}
\end{lemma}

\begin{lemma}
\cite{q bernstein} \label{Lemma 3} Let $q\in \left( 0,1\right) .$ Then for
all $k,j\in \left\{ 0,1,...,n\right\} $ and $x\in \left[ \frac{\left[ j%
\right] _{q}}{\left[ n+1\right] _{q}},\frac{\left[ j+1\right] _{q}}{\left[
n+1\right] _{q}}\right] $, we get 
\begin{equation*}
m_{k,n,j}(x;q)\leq 1.
\end{equation*}
\end{lemma}

\begin{lemma}
\label{Lemma 4} Let $q\in \left( 0,1\right) $ moreover $q=\left(
q_{n}\right) ,$ $n\in 
\mathbb{N}
$ and $\lim_{n\rightarrow \infty }q_{n}=1$, $j\in \left\{ 0,1,...,n\right\} $
and $x\in \left[ \frac{\left[ j\right] _{q}}{\left[ n+1\right] _{q}},\frac{%
\left[ j+1\right] _{q}}{\left[ n+1\right] _{q}}\right] $ and $\alpha
=2,3,... $ . Then we have

$(i)$ if $k\in \left\{ j+2,j+3,...,n-1\right\} $ and $\left[ k+1\right]
_{q}-\left( q^{k}\left[ k+1\right] _{q}\right) ^{\frac{1}{\alpha }}\geq %
\left[ j+1\right] _{q}$ then%
\begin{equation*}
\overline{M}_{k,n,j}(x;q)\geq \overline{M}_{k+1,n,j}(x;q),
\end{equation*}%
$(ii)$ if $k\in \left\{ 1,2,...,j-2\right\} $ and $\left[ k\right]
_{q}+\left( q^{k-1}\left[ k\right] _{q}\right) ^{\frac{1}{\alpha }}\leq %
\left[ j\right] _{q}$ then%
\begin{equation*}
\underline{M}_{k,n,j}(x;q)\geq \underline{M}_{k-1,n,j}(x;q).
\end{equation*}
\end{lemma}

\begin{proof}
$(i)$ From the case $(i)$ of Lemma $4$ in \cite{q bernstein}, we can write%
\begin{equation*}
\frac{\overline{M}_{k,n,j}(x;q)}{\overline{M}_{k+1,n,j}(x;q)}\geq \frac{%
\left[ k+1\right] _{q}}{\left[ j+1\right] _{q}}\frac{\left[ k\right] _{q}-%
\left[ j+1\right] _{q}}{\left[ k+1\right] _{q}-\left[ j+1\right] _{q}}.
\end{equation*}

After this point we will use a different proof tecnique from \cite{q
bernstein}.

By the induction method, let's show that, the following inequality%
\begin{equation}
\frac{\left[ k+1\right] _{q}}{\left[ j+1\right] _{q}}\frac{\left[ k\right]
_{q}-\left[ j+1\right] _{q}}{\left[ k+1\right] _{q}-\left[ j+1\right] _{q}}%
\geq 1  \label{Lemma.4i}
\end{equation}%
holds for $k\geq j+2$ and $\left[ k+1\right] _{q}-\left( q^{k}\left[ k+1%
\right] _{q}\right) ^{\frac{1}{\alpha }}\geq \left[ j+1\right] _{q}.$

For $\alpha =2$, this inequality is $\left[ k+1\right] _{q}-\sqrt{q^{k}\left[
k+1\right] _{q}}^{\frac{1}{\alpha }}\geq \left[ j+1\right] _{q},$ which
becomes the demonstrated case $(i)$ of Lemma $4$ in \cite{q bernstein}. So,
we obtain the inequality (\ref{Lemma.4i}) for $\alpha =2.$

Now, we assume that the inequalty (\ref{Lemma.4i}) is provided for $\alpha
-1.$ This means that (\ref{Lemma.4i}) holds for $k\geq j+2$ and $\left[ k+1%
\right] _{q}-\left( q^{k}\left[ k+1\right] _{q}\right) ^{\frac{1}{\alpha -1}%
}\geq \left[ j+1\right] _{q}.$ It follows that 
\begin{equation*}
\begin{array}{ccc}
\left[ k+1\right] _{q}-\left( q^{k}\left[ k+1\right] _{q}\right) ^{\frac{1}{%
\alpha -1}} & \geq & \left[ j+1\right] _{q} \\ 
\left[ k+1\right] _{q}-\left[ j+1\right] _{q} & \geq & \left( q^{k}\left[ k+1%
\right] _{q}\right) ^{\frac{1}{\alpha -1}} \\ 
\left( \left[ k+1\right] _{q}-\left[ j+1\right] _{q}\right) ^{\alpha -1} & 
\geq & q^{k}\left[ k+1\right] _{q} \\ 
\frac{\left( \left[ k+1\right] _{q}-\left[ j+1\right] _{q}\right) ^{\alpha }%
}{\left[ k+1\right] _{q}-\left[ j+1\right] _{q}} & \geq & q^{k}\left[ k+1%
\right] _{q} \\ 
\left( \left[ k+1\right] _{q}-\left[ j+1\right] _{q}\right) ^{\alpha } & \geq
& \left( q^{k}\left[ k+1\right] _{q}\right) \left( \left[ k+1\right] _{q}-%
\left[ j+1\right] _{q}\right) .%
\end{array}%
\end{equation*}

Also, since $k\geq j+2$ then $\left[ k+1\right] _{q}\geq \left[ j+3\right]
_{q}$, the we have 
\begin{eqnarray*}
\left[ k+1\right] _{q}-\left[ j+1\right] _{q} &\geq &\left[ j+3\right] _{q}-%
\left[ j+1\right] _{q} \\
&=&\frac{1-q^{j+3}}{1-q}-\frac{1-q^{j+1}}{1-q} \\
&=&\frac{q^{j+1}-q^{j+3}}{1-q} \\
&=&\frac{q^{j+1}\left( 1-q^{2}\right) }{1-q} \\
&=&\frac{q^{j+1}\left( 1-q\right) \left( 1+q\right) }{1-q} \\
&=&q^{j+1}\left( 1+q\right) .
\end{eqnarray*}%
Since $\lim_{n\rightarrow \infty }q_{n}=1$ then $q^{j+1}\left( 1+q\right)
\geq 1$. So we obtain $\left[ k+1\right] _{q}-\left[ j+1\right] _{q}\geq 1$.
From the result given above, we get 
\begin{eqnarray*}
\left( \left[ k+1\right] _{q}-\left[ j+1\right] _{q}\right) ^{\alpha } &\geq
&\left( q^{k}\left[ k+1\right] _{q}\right) \left( \left[ k+1\right] _{q}-%
\left[ j+1\right] _{q}\right) \\
&\geq &\left( q^{k}\left[ k+1\right] _{q}\right) .
\end{eqnarray*}%
So, (\ref{Lemma.4i}) is true for $\alpha ,$ hence, for arbitrary $\alpha
=2,3,...,$ the inequality (\ref{Lemma.4i}) is provided when $\left[ k+1%
\right] _{q}-\left( q^{k}\left[ k+1\right] _{q}\right) ^{\frac{1}{\alpha -1}%
}\geq \left[ j+1\right] _{q}.$ So we obtain,%
\begin{equation*}
\dfrac{\overline{M}_{k,n,j}(x;q)}{\overline{M}_{k+1,n,j}(x;q)}\geq \frac{%
\left[ k+1\right] _{q}}{\left[ j+1\right] _{q}}\frac{\left[ k\right] _{q}-%
\left[ j+1\right] _{q}}{\left[ k+1\right] _{q}-\left[ j+1\right] _{q}}\geq 1.
\end{equation*}

$(ii)$ From the case $(ii)$ of Lemma $4$ in \cite{q bernstein}, we can write%
\begin{equation*}
\dfrac{\text{$\underline{M}$}_{k,n,j}(x)}{\text{$\underline{M}$}_{k-1,n,j}(x)%
}\geq \frac{\left[ j\right] _{q}}{\left[ k\right] _{q}}\frac{\left[ j\right]
_{q}-\left[ k\right] _{q}}{\left[ j\right] _{q}-\left[ k-1\right] _{q}}.
\end{equation*}

After this point we will use the our proof technique again. Same as proof of 
$(i),$ using the induction method, let's show that the following inequality 
\begin{equation}
\frac{\left[ j\right] _{q}}{\left[ k\right] _{q}}\frac{\left[ j\right] _{q}-%
\left[ k\right] _{q}}{\left[ j\right] _{q}-\left[ k-1\right] _{q}}\geq 1
\label{Lemma.4ii}
\end{equation}%
holds for $k\leq j-2$ and $\left[ k\right] _{q}+\left( q^{k-1}\left[ k\right]
_{q}\right) ^{\frac{1}{\alpha }}\leq \left[ j\right] _{q}$.

For $\alpha =2$, this inequality is $\left[ k\right] _{q}+\sqrt{q^{k-1}\left[
k\right] _{q}^{\frac{1}{\alpha }}}\leq \left[ j\right] _{q}$, which becomes
the demonstrated case $(ii)$ of Lemma $4$ in \cite{q bernstein}. So, we
obtain the inequality (\ref{Lemma.4ii}) is satisfied.

Now, we assume that (\ref{Lemma.4ii}) is correct for $\alpha -1.$ This means
that the inequality (\ref{Lemma.4ii}) is provided for $k\leq j-2$ and $\left[
k\right] _{q}+\left( q^{k-1}\left[ k\right] _{q}\right) ^{\frac{1}{\alpha -1}%
}\leq \left[ j\right] _{q}$. It follows%
\begin{equation*}
\begin{array}{ccc}
\left[ k\right] _{q}+\left( q^{k-1}\left[ k\right] _{q}\right) ^{\frac{1}{%
\alpha -1}} & \leq & \left[ j\right] _{q} \\ 
\left( q^{k-1}\left[ k\right] _{q}\right) ^{\frac{1}{\alpha -1}} & \leq & 
\left[ j\right] _{q}-\left[ k\right] _{q} \\ 
q^{k-1}\left[ k\right] _{q} & \leq & \left( \left[ j\right] _{q}-\left[ k%
\right] _{q}\right) ^{\alpha -1} \\ 
\left( q^{k-1}\left[ k\right] _{q}\right) \left( \left[ j\right] _{q}-\left[
k\right] _{q}\right) & \leq & \left( \left[ j\right] _{q}-\left[ k\right]
_{q}\right) ^{\alpha }.%
\end{array}%
\end{equation*}%
Also, since $k\leq j-2$ then $k+2\leq j,$ $\left[ k+1\right] _{q}\geq \left[
j+3\right] _{q}$. So we have%
\begin{eqnarray*}
\left[ j\right] _{q}-\left[ k\right] _{q} &\geq &\left[ k+2\right] _{q}-%
\left[ k\right] _{q} \\
&=&\frac{1-q^{k+2}}{1-q}-\frac{1-q^{k}}{1-q} \\
&=&\frac{q^{k}-q^{k+2}}{1-q} \\
&=&\frac{q^{k}\left( 1-q^{2}\right) }{1-q} \\
&=&q^{k}\left( 1+q\right) .
\end{eqnarray*}%
Since $\lim_{n\rightarrow \infty }q_{n}=1$ then $q^{k}\left( 1+q\right) \geq
1$. It follows $\left[ j\right] _{q}-\left[ k\right] _{q}\geq 1$ and we
obtain $q^{k-1}\left[ k\right] _{q}\leq \left( q^{k-1}\left[ k\right]
_{q}\right) \left( \left[ j\right] _{q}-\left[ k\right] _{q}\right) \leq
\left( \left[ j\right] _{q}-\left[ k\right] _{q}\right) ^{\alpha }.$Which
gives the desired result.
\end{proof}

\begin{lemma}
\cite{q bernstein} \label{Lemma 5} Let $q\in \left( 0,1\right) ,$ $j\in
\left\{ 0,1,...,n\right\} $ and $x\in \left[ \frac{\left[ j\right] _{q}}{%
\left[ n+1\right] _{q}},\frac{\left[ j+1\right] _{q}}{\left[ n+1\right] _{q}}%
\right] $. Then we get 
\begin{equation*}
\bigvee\limits_{k=0}^{n}p_{n,k}\left( x;q\right) =p_{n,j}(x;q).
\end{equation*}
\end{lemma}

\section{Approximation Results}

The main aim of this section is to obtain a better order of approximation
for the operators $B_{n,q}^{(M)}(f)\left( x\right) $ to the function $f$ by
means of the modulus of continuity. According to the following theorem we
can say that the order of approximation can be improved when the $\alpha $
is big enough. Moreover if we choose as $\alpha =2$, these results turn out
to be the results in \cite{q bernstein}.

\begin{theorem}
\label{Theorem 6} If $f:\left[ 0,1\right] \rightarrow 
\mathbb{R}
_{+}$ is continuous, then we have the following order of approximation for
the $B_{n,q}^{\left( M\right) }\left( f\right) \left( x\right) $ to the
function $f$ by means of the modulus of continuity:%
\begin{equation*}
\left\vert B_{n,q}^{(M)}(f)\left( x\right) -f\left( x\right) \right\vert
\leq 4\left( 1+\frac{2}{q^{n+1}}\right) \omega \left( f;\frac{1}{\left[ n+1%
\right] _{q}^{1-\frac{1}{\alpha }}}\right) ,\text{ for all }n\in 
\mathbb{N}
,x\in \left[ 0,1\right]
\end{equation*}%
where $\alpha =2,3,...$ and $q\in \left( 0,1\right) $ moreover $q=\left(
q_{n}\right) ,$ $n\in 
\mathbb{N}
$ and $\lim_{m\rightarrow \infty }q_{n}=1$.
\end{theorem}

\begin{proof}
Since nonlinear max-product $q$-Bernstein operators satisfy the conditions
in Corollary \ref{Corollary 1}, we get%
\begin{equation}
\left\vert B_{n,q}^{(M)}(f)\left( x\right) -f\left( x\right) \right\vert
\leq 2\left[ 1+\frac{1}{\delta _{n}}B_{n,q}^{(M)}\left( \varphi _{x}\right)
\left( x\right) \right] \,\omega \left( f,\delta \right) ,  \label{teo1}
\end{equation}%
where $\varphi _{x}\left( t\right) =\left\vert t-x\right\vert .$ At this
point let us denote%
\begin{equation*}
E_{n,q}\left( x\right) :=B_{n,q}^{(M)}\left( \varphi _{x}\right) \left(
x\right) =\frac{\bigvee\limits_{k=0}^{n}p_{n,k}\left( x;q\right) \left\vert 
\frac{\left[ k\right] _{q}}{\left[ n\right] _{q}}-x\right\vert }{%
\bigvee\limits_{k=0}^{n}p_{n,k}\left( x;q\right) },\text{ }x\in \left[ 0,1%
\right] .
\end{equation*}%
Let $x\in \left[ \frac{\left[ j\right] _{q}}{\left[ n+1\right] _{q}},\frac{%
\left[ j+1\right] _{q}}{\left[ n+1\right] _{q}}\right] ,$ where $j\in
\left\{ 0,1,...,n\right\} $ is fixed, arbitary. By Lemma \ref{Lemma 5} we
can easily obtain 
\begin{equation*}
E_{n,q}\left( x\right) =\max\limits_{k=0,1,...,n}\left\{ M_{k,n,j}\left(
x;q\right) \right\} ,\text{ }x\in \left[ \frac{\left[ j\right] _{q}}{\left[
n+1\right] _{q}},\frac{\left[ j+1\right] _{q}}{\left[ n+1\right] _{q}}\right]
.
\end{equation*}%
It can be examined, from Corollary $6$ in \cite{q bernstein}, that $%
M_{k,n,0}\left( x;q\right) \leq \frac{1}{\left[ n+1\right] _{q}},$ for $%
k=0,1,2,...,n$ and $j=0,$ where $x\in \left[ 0,\frac{1}{\left[ n+1\right]
_{q}}\right] .$ So, we have an upper estimate for any $k=0,1,...,n,$ $%
E_{n,q}\left( x\right) \leq \frac{1}{\left[ n+1\right] _{q}}$ when $j=0.$

Now, it remains to find an upper estimate for each $M_{k,n,j}\left( x\right) 
$ when $j=1,2,...,n$, is fixed, $x\in \left[ \frac{\left[ j\right] _{q}}{%
\left[ n+1\right] _{q}},\frac{\left[ j+1\right] _{q}}{\left[ n+1\right] _{q}}%
\right] $, $k\in \left\{ 0,1,...,n\right\} $ and $\alpha =2,3,...$ . In fact
we will prove that 
\begin{equation}
M_{k,n,j}\left( x\right) \leq \frac{2\left( 1+\frac{2}{q^{n+1}}\right) }{%
\left[ n+1\right] _{q}^{1-\frac{1}{\alpha }}},  \label{teo2}
\end{equation}%
for all $x\in \left[ \frac{\left[ j\right] _{q}}{\left[ n+1\right] _{q}},%
\frac{\left[ j+1\right] _{q}}{\left[ n+1\right] _{q}}\right] ,$ $%
k=0,1,2,...,n$ which directly will implies that 
\begin{equation*}
E_{n,q}\left( x\right) \leq \frac{2\left( 1+\frac{2}{q^{n+1}}\right) }{\left[
n+1\right] _{q}^{1-\frac{1}{\alpha }}},\text{ for all }x\in \left[ 0,1\right]
,\text{ }n\in \mathbb{N}
\end{equation*}%
and choosing $\delta _{n}=\frac{2\left( 1+\frac{2}{q^{n+1}}\right) }{\left[
n+1\right] _{q}^{1-\frac{1}{\alpha }}}$ in (\ref{teo1}) we obtain the
estimate in the statement immediately.

So, in order to completing the proof of (\ref{teo2}), we consider the
following cases:

$1)$ $k\in \left\{ j-1,j,j+1\right\} $,

$2)$ $k\geq j+2,$

and

$3)$ $k\leq j-2.$\newline

Case $1).$ If $k=j-1,$ then $M_{j-1,n,j}\left( x;q\right) =m_{j-1,n,j}\left(
x;q\right) \left( x-\frac{\left[ j-1\right] _{q}}{\left[ n\right] _{q}}%
\right) .$ Since by Lemma \ref{Lemma 3}, we get 
\begin{eqnarray*}
M_{j-1,n,j}\left( x;q\right) &\leq &x-\frac{\left[ j-1\right] _{q}}{\left[ n%
\right] _{q}} \\
&\leq &\frac{\left[ j+1\right] _{q}}{\left[ n+1\right] _{q}}-\frac{\left[ j-1%
\right] _{q}}{\left[ n+1\right] _{q}} \\
&=&\frac{q^{j-1}\left( 1+q\right) }{\left[ n+1\right] _{q}}\leq \frac{2}{%
\left[ n+1\right] _{q}}.
\end{eqnarray*}

If $k=j,$ then since $x\in \left[ \frac{\left[ j\right] _{q}}{\left[ n+1%
\right] _{q}},\frac{\left[ j+1\right] _{q}}{\left[ n+1\right] _{q}}\right] $
we get 
\begin{eqnarray*}
M_{j,n,j}\left( x;q\right) &=&\frac{p_{n,j}\left( x\right) \left\vert \frac{%
\left[ j\right] _{q}}{\left[ n\right] _{q}}-x\right\vert }{p_{n,j}\left(
x\right) }=\left\vert \frac{\left[ j\right] _{q}}{\left[ n\right] _{q}}%
-x\right\vert \leq \left\vert \frac{\left[ j\right] _{q}}{\left[ n\right]
_{q}}-\frac{\left[ j\right] _{q}}{\left[ n+1\right] _{q}}\right\vert \\
&=&\frac{\left[ j\right] _{q}q^{n}}{\left[ n\right] _{q}\left[ n+1\right]
_{q}}\leq \frac{1}{\left[ n+1\right] _{q}}.
\end{eqnarray*}

If $k=j+1,$ then $M_{j+1,n,j}\left( x;q\right) =m_{j+1,n,j}\left( x;q\right)
\left( \frac{\left[ j+1\right] _{q}}{\left[ n\right] _{q}}-x\right) .$ Since
by Lemma \ref{Lemma 3} we get 
\begin{eqnarray*}
M_{j+1,n,j}\left( x;q\right) &\leq &\frac{\left[ j+1\right] _{q}}{\left[ n%
\right] _{q}}-x \\
&\leq &\frac{\left[ j+1\right] _{q}}{\left[ n\right] _{q}}-\frac{\left[ j%
\right] _{q}}{\left[ n+1\right] _{q}} \\
&=&\frac{\left[ j+1\right] _{q}\left[ n+1\right] _{q}-\left[ j\right] _{q}%
\left[ n\right] _{q}}{\left[ n\right] _{q}\left[ n+1\right] _{q}} \\
&\leq &\frac{3}{\left[ n+1\right] _{q}}.
\end{eqnarray*}

Case $2).$ Subcase $a).$ Firstly assume that $k\geq j+2$ and $\left[ k+1%
\right] _{q}-\left( q^{k}\left[ k+1\right] _{q}\right) ^{\frac{1}{\alpha }}<%
\left[ j+1\right] _{q}.$ Since the hypothesis also $q\left[ k\right]
_{q}-\left( q^{k}\left[ k+1\right] _{q}\right) ^{\frac{1}{\alpha }}<q\left[ j%
\right] _{q}$ and since Lemma \ref{Lemma 3} we get%
\begin{eqnarray*}
\overline{M}_{k,n,j}\left( x;q\right) &=&m_{k,n,j}\left( x;q\right) \left( 
\frac{\left[ k\right] _{q}}{\left[ n+1\right] _{q}}-x\right) \\
&\leq &\frac{\left[ k\right] _{q}}{\left[ n+1\right] _{q}}-\frac{\left[ j%
\right] _{q}}{\left[ n+1\right] _{q}} \\
&\leq &\frac{\left[ k\right] _{q}}{\left[ n+1\right] _{q}}-\frac{\left[ k%
\right] _{q}-\frac{1}{q}\left( q^{k}\left[ k+1\right] _{q}\right) ^{\frac{1}{%
\alpha }}}{\left[ n+1\right] _{q}} \\
&=&\frac{\left( q^{k-\alpha }\left[ k+1\right] _{q}\right) ^{\frac{1}{\alpha 
}}}{\left[ n+1\right] _{q}}\leq \frac{\left( q^{k-\alpha }\left[ n+1\right]
_{q}\right) ^{\frac{1}{\alpha }}}{\left[ n+1\right] _{q}} \\
&\leq &\frac{1}{\left[ n+1\right] _{q}^{1-\frac{1}{\alpha }}},
\end{eqnarray*}%
where $k\geq \alpha .$ Because, if $k<\alpha $ then $q^{k-\alpha }=\frac{1}{%
q^{\alpha -k}}\geq 1.$

Subcase $b).$ Assume now that $k\geq j+2$ and $\left[ k+1\right] _{q}-\left(
q^{k}\left[ k+1\right] _{q}\right) ^{\frac{1}{\alpha }}\geq \left[ j+1\right]
_{q}.$ Let's define a function $g_{\alpha ,q}$ as $g_{\alpha ,q}\left(
k\right) :=\left[ k+1\right] _{q}-\left( q^{k}\left[ k+1\right] _{q}\right)
^{\frac{1}{\alpha }}$. It can easily be shown that the function $g_{\alpha
,q}$ is increasing on the interval $\left[ 0,1\right] .$ To show this, after
simple calculations, we get%
\begin{equation*}
g_{\alpha ,q}\left( k+1\right) -g_{\alpha ,q}\left( k\right) \geq q^{k+1}-q^{%
\frac{k}{\alpha }}\left( \left[ k+2\right] _{q}^{\frac{1}{\alpha }}-\left[
k+1\right] _{q}^{\frac{1}{\alpha }}\right) .
\end{equation*}%
Now, since $\alpha =2,3,...$ then if we denote the number $H_{\alpha ,q}$ as 
$H_{\alpha ,q}:=\left( \left[ k+2\right] _{q}^{\alpha -1}\right) ^{\frac{1}{%
\alpha }}+\left( \left[ k+2\right] _{q}^{\alpha -2}\left[ k+1\right]
_{q}\right) ^{\frac{1}{\alpha }}+...+\left( \left[ k+2\right] _{q}\left[ k+1%
\right] _{q}^{\alpha -2}\right) ^{\frac{1}{\alpha }}+\left( \left[ k+1\right]
_{q}^{\alpha -1}\right) ^{\frac{1}{\alpha }}$, Then we get $H_{\alpha
,q}\geq 1$. Therefore, we have 
\begin{eqnarray*}
g\left( k+1\right) -g\left( k\right) &\geq &q^{k+1}-q^{\frac{k}{\alpha }%
}\left( \left[ k+2\right] _{q}^{\frac{1}{\alpha }}-\left[ k+1\right] _{q}^{%
\frac{1}{\alpha }}\right) \\
&=&q^{k+1}-q^{\frac{k}{\alpha }}\left( \left[ k+2\right] _{q}^{\frac{1}{%
\alpha }}-\left[ k+1\right] _{q}^{\frac{1}{\alpha }}\right) \frac{H_{\alpha }%
}{H_{\alpha }} \\
&=&q^{k+1}-q^{\frac{k}{\alpha }}\frac{\left[ k+2\right] _{q}-\left[ k+1%
\right] _{q}}{H_{\alpha }} \\
&=&q^{k+1}-\frac{q^{\frac{k}{\alpha }}q^{k+1}}{H_{\alpha ,q}}=q^{k+1}\left(
1-\frac{q^{\frac{k}{\alpha }}}{H_{\alpha ,q}}\right) \\
&\geq &q^{k+1}\left( 1-\frac{1}{H_{\alpha ,q}}\right) >0.
\end{eqnarray*}%
It follows that there exists a maximum value $\bar{k}\in \left\{
1,2,...,n\right\} $, satisfying the inequality $\left[ \bar{k}+1\right]
_{q}-\left( q^{\bar{k}}\left[ \bar{k}+1\right] _{q}\right) ^{\frac{1}{\alpha 
}}<\left[ j+1\right] _{q}.$ Let $\tilde{k}=\bar{k}+1,$ then for all $k\geq 
\tilde{k},$ we have $\left[ k+1\right] _{q}-\left( q^{k}\left[ k+1\right]
_{q}\right) ^{\frac{1}{\alpha }}\geq \left[ j+1\right] _{q}.$Since $%
g_{\alpha ,q}$ is increasing, and because the hypothesis is $k\geq j+2$ it
is easy to see that $\tilde{k}>j+1$ and so $\tilde{k}\geq j+2.$ Also, from %
\ref{Lemma 3} we have%
\begin{eqnarray*}
\overline{M}_{\tilde{k}+1,n,j}\left( x;q\right) &=&m_{\tilde{k}+1,n,j}\left(
x\right) \left( \frac{\left[ \tilde{k}+1\right] _{q}}{\left[ n+1\right] _{q}}%
-x\right) \\
&\leq &\frac{\left[ \bar{k}+1\right] _{q}}{\left[ n+1\right] _{q}}-x\leq 
\frac{\left[ \bar{k}+1\right] _{q}}{\left[ n+1\right] _{q}}-\frac{\left[ j%
\right] _{q}}{\left[ n+1\right] _{q}}.
\end{eqnarray*}%
Thus, since 
\begin{equation*}
\left[ j\right] _{q}>\left[ \bar{k}+1\right] _{q}-q^{j}-\left( q^{\bar{k}}%
\left[ \bar{k}+1\right] _{q}\right) ^{\frac{1}{\alpha }},
\end{equation*}%
we can write%
\begin{eqnarray*}
\overline{M}_{\tilde{k}+1,n,j}\left( x;q\right) &\leq &\frac{\left[ \bar{k}+1%
\right] _{q}}{\left[ n+1\right] _{q}}-\frac{\left[ \bar{k}+1\right]
_{q}-q^{j}-\left( q^{\bar{k}}\left[ \bar{k}+1\right] _{q}\right) ^{\frac{1}{%
\alpha }}}{\left[ n+1\right] _{q}} \\
&=&\frac{q^{j}+\left( q^{\bar{k}}\left[ \bar{k}+1\right] _{q}\right) ^{\frac{%
1}{\alpha }}}{\left[ n+1\right] _{q}}\leq \frac{q^{j}+\left[ n+1\right]
_{q}^{\frac{1}{\alpha }}}{\left[ n+1\right] _{q}} \\
&\leq &\frac{2\left[ n+1\right] _{q}^{\frac{1}{\alpha }}}{\left[ n+1\right]
_{q}}=\frac{2}{\left[ n+1\right] _{q}^{1-\frac{1}{\alpha }}}.
\end{eqnarray*}%
Lemma \ref{Lemma 4}, $(i),$ indicates that 
\begin{equation*}
\overline{M}_{\tilde{k}+1,n,j}\left( x;q\right) \geq \overline{M}_{\tilde{k}%
+2,n,j}\left( x;q\right) \geq ...\geq \overline{M}_{n,n,j}\left( x;q\right)
\end{equation*}%
is satisfied. Thus, we obtain $\overline{M}_{k,n,j}\left( x;q\right) \leq 
\frac{2}{\,\left[ n+1\right] _{q}^{1-\frac{1}{\alpha }}}$ for all $k\in
\left\{ \tilde{k}+1,\tilde{k}+2,...,n\right\} $.

Thereby, we obtain in both subcases, by Lemma \ref{Lemma 2}, $(i)$ too, $%
M_{k,n,j}\left( x;q\right) \leq \frac{2\left( 1+\frac{2}{q^{n+1}}\right) }{\,%
\left[ n\right] _{q}^{1--\frac{1}{\alpha }}}$.

Case $3).$ Subcase $a).$ Assume first that $\left[ k\right] _{q}+\left(
q^{k-1}\left[ k\right] _{q}\right) ^{\frac{1}{\alpha }}\geq \left[ j\right]
_{q}$. Then we have%
\begin{eqnarray*}
\text{$\underline{M}$}_{k,n,j}\left( x;q\right) &=&m_{k,n,j}\left(
x;q\right) \left( x-\frac{\left[ k\right] _{q}}{\left[ n+1\right] _{q}}%
\right) \\
&\leq &\frac{\left[ j+1\right] _{q}}{\left[ n+1\right] _{q}}-\frac{\left[ k%
\right] _{q}}{\left[ n+1\right] _{q}}=\frac{\left[ j\right] _{q}+q^{j}}{%
\left[ n+1\right] _{q}}-\frac{\left[ k\right] _{q}}{\left[ n+1\right] _{q}}
\\
&\leq &\frac{\left[ k\right] _{q}+\left( q^{k-1}\left[ k\right] _{q}\right)
^{\frac{1}{\alpha }}+q^{j}}{\left[ n+1\right] _{q}}-\frac{\left[ k\right]
_{q}}{\left[ n+1\right] _{q}} \\
&=&\frac{\left( q^{k-1}\left[ k\right] _{q}\right) ^{\frac{1}{\alpha }}+q^{j}%
}{\left[ n+1\right] _{q}}\leq \frac{\left[ n+1\right] _{q}^{\frac{1}{\alpha }%
}+1}{\left[ n+1\right] _{q}} \\
&\leq &\frac{2}{\,\left[ n+1\right] _{q}^{1-\frac{1}{\alpha }}}.
\end{eqnarray*}

Subcase $b).$ Assume now that $k\leq j-2$ and $\left[ k\right] _{q}+\left(
q^{k-1}\left[ k\right] _{q}\right) ^{\frac{1}{\alpha }}<\left[ j\right] _{q}$%
. Let $\bar{k}\in \left\{ 0,1,2,...,n\right\} $ be the minimum value such
that $\left[ \bar{k}\right] _{q}+\left( q^{\bar{k}-1}\left[ \bar{k}\right]
_{q}\right) ^{\frac{1}{\alpha }}\geq \left[ j\right] _{q}$. Let $\tilde{k}=%
\bar{k}-1.$ We can write that $\left[ \bar{k}-1\right] _{q}+\left( q^{\bar{k}%
-2}\left[ \bar{k}-1\right] _{q}\right) ^{\frac{1}{\alpha }}<\left[ j\right]
_{q}.$ Because the hypothesis is $k\leq j-2$ it is easy to see that $\tilde{k%
}\leq j-2.$ Also, we get%
\begin{eqnarray*}
\text{$\underline{M}$}_{\bar{k}-1,n,j}\left( x;q\right) &=&m_{\bar{k}%
-1,n,j}\left( x;q\right) \left( x-\frac{\left[ \bar{k}-1\right] _{q}}{\left[
n+1\right] _{q}}\right) \\
&\leq &x-\frac{\left[ \bar{k}-1\right] _{q}}{\left[ n+1\right] _{q}}\leq 
\frac{\left[ j+1\right] _{q}}{\left[ n+1\right] _{q}}-\frac{\left[ \bar{k}-1%
\right] _{q}}{\left[ n+1\right] _{q}} \\
&\leq &\frac{\left[ j\right] _{q}+q^{j}}{\left[ n+1\right] _{q}}-\frac{\left[
\bar{k}-1\right] _{q}}{\left[ n+1\right] _{q}}.
\end{eqnarray*}%
Since%
\begin{equation*}
\left[ j\right] _{q}\leq \left[ \bar{k}\right] _{q}+\left( q^{\bar{k}-1}%
\left[ \bar{k}\right] _{q}\right) ^{\frac{1}{\alpha }},
\end{equation*}%
we see that%
\begin{eqnarray*}
\text{$\underline{M}$}_{\tilde{k}-1,n,j}\left( x;q\right) &\leq &\frac{\left[
\bar{k}\right] _{q}+q^{j}+\left( q^{\bar{k}-1}\left[ \bar{k}\right]
_{q}\right) ^{\frac{1}{\alpha }}}{\left[ n+1\right] _{q}}-\frac{\left[ \bar{k%
}-1\right] _{q}}{\left[ n+1\right] _{q}} \\
&=&\frac{q^{j}+q^{\bar{k}-1}+\left( q^{\bar{k}-1}\left[ \bar{k}\right]
_{q}\right) ^{\frac{1}{\alpha }}}{\left[ n+1\right] _{q}}\leq \frac{2+\left[
n+1\right] _{q}^{\frac{1}{\alpha }}}{\left[ n+1\right] _{q}} \\
&\leq &\frac{3}{\left( \left[ n+1\right] _{q}\right) ^{1-\frac{1}{\alpha }}}%
\,.
\end{eqnarray*}%
By Lemma \ref{Lemma 4}, $(ii)$ , it follows that $\underline{M}_{\tilde{k}%
-1,n,j}\left( x;q\right) \geq \underline{M}_{\tilde{k}-2,n,j}\left(
x;q\right) \geq ...\geq \underline{M}_{0,n,j}\left( x;q\right) .$ Thus we
obtain for all $k\in \left\{ 0,1,...,\tilde{k}-1,\tilde{k}\right\} $ 
\begin{equation*}
\underline{M}_{k,n,j}\left( x;q\right) \leq \frac{3}{\left( \left[ n+1\right]
_{q}\right) ^{1-\frac{1}{\alpha }}}.
\end{equation*}%
Therefore, in both subcases, by Lemma \ref{Lemma 2}, $(ii)$, we get%
\begin{equation*}
M_{k,n,j}\left( x;q\right) \leq \frac{3}{\,\left( \left[ n+1\right]
_{q}\right) ^{1-\frac{1}{\alpha }}}.
\end{equation*}%
As a result, the inequation (\ref{teo2}) is satisfied for all $k,j\in
\left\{ 0,1,2,...,n\right\} ,$ $n\in 
\mathbb{N}
$ and $x\in \left[ 0,1\right] .$ So, we have desired result.
\end{proof}

\begin{corollary}
In \cite{q bernstein}, the order of approximation for nonlinear max-product $%
q$-Bernstein operators was found as $1/\sqrt{\left[ n\right] _{q}}$ by means
of modulus of continuity. However, due to Theorem \ref{Theorem 6}, we proved
that the order of approximation is $1/\left[ n+1\right] _{q}^{1-\frac{1}{%
\alpha }}.$ For big enough $\alpha $, $1/\left[ n+1\right] _{q}^{1-\frac{1}{%
\alpha }}$ tends to $1/\left[ n+1\right] _{q}$. As a result, since $1-\frac{1%
}{\alpha }\geq \frac{1}{2}$ for $\alpha =2,3,...,$ this selection of $\alpha 
$ improving the order of approximation.
\end{corollary}

\section{$A$-Statistical Approximation}

Firstly, let's recall some of the statistical convergence concepts that Fast
examined in \cite{fast}. Let's consider that the $A:=\left( a_{nk}\right) $
is infinite summability matrix. For a given sequence $x=\left( x_{k}\right) $%
, the $A$-transform of $x$, denoted by $\left( Ax=\left( Ax\right)
_{n}\right) $, is defined as 
\begin{equation*}
\left( Ax\right) _{n}:=\sum_{k=1}^{\infty }a_{nk}x_{k},
\end{equation*}%
provided the series that converges for each $n\cdot A$ is said to be regular
if $\lim_{n}(Ax)_{n}=L$ whenever $\lim x=L$ \cite{hardy divergent}. Assume
that $A$ is a regular summability matrix that is non-negative. Then $x$ is $%
A $-statistically convergent to $L$ if for every $\epsilon >0$,%
\begin{equation*}
\lim_{n}\sum_{k:\left\vert x_{k}-L\right\vert \geq \epsilon }a_{nk}=0
\end{equation*}

and we write $st_{A}-\lim x=L$ \cite{freedman}. In actuality, $x$ is $A$%
-statistically convergent to $L$ if and only if, for every $\epsilon >0$, $%
\delta _{A}\left( k\in 
\mathbb{N}
:\left\vert x_{k}-L\right\vert \geq \epsilon \right) =0$, where $\delta
_{A}\left( K\right) $ indicates the $A$-density of the subset $K$ of the
natural numbers, and is given by $\delta _{A}\left( K\right)
:=\lim_{n}\sum_{k=1}^{\infty }a_{nk}\chi _{K}\left( k\right) ,$ provided the
limit exists, where $\chi _{K}$ is the characteristic function of $K$. $A$%
-statistical convergence simplifies to statistical convergence if $A=C_{1}$,
the Ces\'{a}ro matrix of order one \cite{fast}. Moreover, $A$-statistical
convergence and ordinary convergence coincide when $A=I$, the identity
matrix, is used.

Using the concept of statistical convergence, Gadjiev and Orhan \cite%
{Gadjiev} proved the Korovkin type the conclusion for any sequence of
positive linear operators; nevertheless, it also holds true for $A$%
-statistical convergence (see \cite{Duman A-ist}). Notice that Duman
examined the max-product operators' statistical convergence in \cite{duman
max-prod st conver}.

Now, let $A=\left( a_{nk}\right) $ be a nonnegative regular summability
matrix. Then replacing $q$ in (\ref{q-Berns op}) by a sequence $\left(
q_{n}\right) ,$

\begin{equation}
0<q_{n}<1\text{, }st_{A}-\lim_{n}q_{n}^{n}=1\text{ and }st_{A}-\lim_{n}\frac{%
1}{\left[ n\right] _{q_{n}}}=0.  \label{st-limit q}
\end{equation}%
For example, take $A=C_{1}$, the Ces\'{a}ro matrix of order one, and define
the sequence $\left( q_{n}\right) $ by%
\begin{equation*}
q_{n}=\left\{ 
\begin{array}{ccc}
0, & if\text{ \ }n=m^{2}, & \left( m=1,2,...\right) \\ 
1-\frac{e^{-n}}{n}, & if\text{ \ }n\neq m^{2}. & 
\end{array}%
\right.
\end{equation*}

Since $1\geq \left( 1-\frac{e^{-n}}{n}\right) ^{n}\geq 1-e^{-n}$, and the $%
C_{1}$- density (or natural density) of the set of all squares is zero, $%
st_{A}-\lim_{n}q_{n}^{n}=1$. On the other hand, if $n\neq m^{2}$ then, for $%
r=0,1,...,n-1,$ $q_{n}^{r}=\left( 1-\frac{e^{-n}}{n}\right) ^{r}\geq 1-r%
\frac{e^{-n}}{n}$. Hence, if $n\neq m^{2}$ then, $\left[ n\right]
_{q_{n}}=1+q_{n}+...+q_{n}^{n-1}\geq n-\frac{n\left( n-1\right) }{2}\frac{%
e^{-n}}{n}.$ This guarantees that $st_{A}-\lim_{n}\frac{1}{\left[ n\right]
_{q_{n}}}=0$. But, $\left( q_{n}\right) $ does not ordinary converge.

Additionally, statistical convergence of max-product $q$-Bernstein operators
was studied by Duman in \cite{q bernstein}. However, since we found a better
order of approximation in Theorem \ref{Theorem 6} than in Theorem $8$ in 
\cite{q bernstein}, we will improve the approximation with the $A$%
-statistical convergence theorem we will give in this section.

\begin{theorem}
\label{Theorem 8} Let $A=\left( a_{nk}\right) $ be a non-negative regular
summability matrix and $\left( q_{n}\right) $ be a sequence satisfying (\ref%
{st-limit q}). Then for every $f\in C_{+}\left[ 0,1\right] $ we obtain 
\begin{equation*}
st_{A}-\lim_{n}\left( \sup_{x\in \left[ 0,1\right] }\left\vert
B_{n,q}^{\left( M\right) }\left( f\right) \left( x\right) -f\left( x\right)
\right\vert \right) =0.
\end{equation*}
\end{theorem}

\begin{proof}
Let $f\in C_{+}\left[ 0,1\right] .$ Using the monotonicity of the modulus of
continuity, supremum over $x\in \left[ 0,1\right] $, and replacing $q$ with $%
\left( q_{n}\right) $, we get from Theorem \ref{Theorem 6} that for every $%
n\in 
\mathbb{N}
,$%
\begin{equation}
E_{n}^{q}:=\sup_{x\in \left[ 0,1\right] }\left\vert B_{n,q_{n}}\left(
f\right) \left( x\right) -f\left( x\right) \right\vert \leq 4\left( 1+\frac{2%
}{q_{n}^{n+1}}\right) \omega \left( f;\frac{1}{\left[ n\right] _{q_{n}}^{1-%
\frac{1}{\alpha }}}\right) .  \label{(teo 8.1)}
\end{equation}%
So, it is enough to prove 
\begin{equation*}
st_{A}-\lim_{n}E_{n}^{q}=0.
\end{equation*}%
The condition in hypothesis (\ref{st-limit q}) implies that%
\begin{equation*}
st_{A}-\lim_{n}\frac{1}{\left[ n\right] _{q_{n}}^{1-\frac{1}{\alpha }}}=0%
\text{.}
\end{equation*}%
Then we have 
\begin{equation}
st_{A}-\lim_{n}\omega \left( f;\frac{1}{\left[ n\right] _{q_{n}}^{1-\frac{1}{%
\alpha }}}\right) =0  \label{(teo 8.2)}
\end{equation}%
Therefore, the proof is completed from (\ref{st-limit q}), (\ref{(teo 8.1)})
and (\ref{(teo 8.2)}).
\end{proof}

Notice that $A$-statistical approximation result includes the classical ones
by choosing as the identity matrix $I.$

\textbf{Supporting/Supporting Organizations:} No grants were received from
any public, private or non-profit organizations for this research.

\textbf{Author's contributions:} All authors contributed equally to the
writing of this paper. All authors read and approved the final manuscript.

\textbf{Availability of data and materials:} Not applicable.

\end{document}